\documentclass[a4paper]{article}

\usepackage{amsmath,amsthm,amssymb}
\usepackage{comment}
\usepackage{tikz}
\usetikzlibrary{calc,patterns,decorations.pathreplacing}
\usepackage{hyperref}

\newcommand{\N}{\mathbb{N}}
\newcommand{\Z}{\mathbb{Z}}
\newcommand{\R}{\mathbb{R}}

\newcommand{\drawhex}[3][white]{\draw[fill=#1] (#2*2*cos{60}+#2*2,#3*2*sin{60}) -- ++(1,0) -- ++(60:1cm) -- ++(120:1cm) -- ++(-1,0) -- ++(240:1cm) -- cycle;}
\newcommand{\drawhexi}[3][white]{\draw[fill=#1] (#2*2*cos{60}+#2*2+1+cos{60},#3*2*sin{60}+sin{60}) -- ++(1,0) -- ++(60:1cm) -- ++(120:1cm) -- ++(-1,0) -- ++(240:1cm) -- cycle;}
\newcommand{\poshex}[2]{#1*2*cos{60}+#1*2+0.5,#2*2*sin{60}+sin{60}}
\newcommand{\poshexi}[2]{#1*2*cos{60}+#1*2+1.5+cos{60},#2*2*sin{60}+2*sin{60}}

\pgfmathsetmacro\sins{sin{60}}
\pgfmathsetmacro\coss{cos{60}}

\newtheorem{lemma}{Lemma}
\newtheorem{corollary}{Corollary}
\newtheorem{theorem}{Theorem}
\newtheorem{problem}{Problem}

\theoremstyle{remark}
\newtheorem{remark}{Remark}

\title{Infinite Hex is arithmetic}

\author{Ilkka Törmä \\ Department of Mathematics and Statistics \\ University of Turku \\ Turku, Finland \\ \href{mailto:iatorm@utu.fi}{\texttt{iatorm@utu.fi}}}

\begin{document}
\maketitle

\begin{abstract}
  Hex is a well known connection game in which two players attempt to connect opposite sides of the board by colored stones.
  In 2022, Hamkins and Leonessi introduced an infinite version, in which the goal is to construct a certain kind of two-way infinite path of adjacent stones.
  It was explicitly left open whether the winning condition is Borel.
  We prove that it is arithmetic, with complexity between $\Sigma^0_4$ and $\Delta^0_5$.
\end{abstract}

\section{Introduction}

Hex is a strategy game played on a rhombus-shaped board of hexagonal tiles.
Two players, black and white, alternate placing one stone of their respective color onto some vacant tile.
Two opposite sides of the board are colored black and the other two white, and the goal of each player is to form a path of adjacent stones connecting their two sides.
Figure~\ref{fig:finite-hex} depicts the final position in a game of Hex played on a $9 \times 9$ board.
Hex cannot end in a draw: a completely filled board is a win for exactly one player \cite{Ga79}.
A strategy-stealing argument shows that the first player is the one with a winning strategy \cite{Na52}.

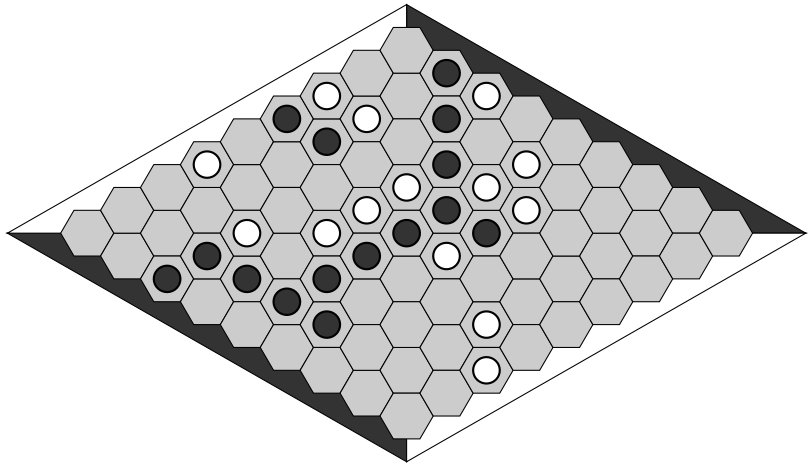
\begin{figure}[htp]
\begin{center}
\begin{tikzpicture}[scale=0.35,rotate=-60]
\draw[fill=black!80] (\poshexi{4}{-1}) -- (\poshexi{-1}{-6}) -- (\poshexi{4}{9}) -- (\poshexi{-1}{4}) -- cycle;
\draw[fill=black!0] (\poshexi{4}{-1}) -- (\poshexi{4}{9}) -- (\poshexi{-1}{-6}) -- (\poshexi{-1}{4}) -- cycle;

\foreach \x in {0,...,3}{
  \pgfmathsetmacro\yb{\x-4}
  \pgfmathsetmacro\yt{\x+4}
  \foreach \y in {\yb,...,\yt}{
    \drawhex[black!20]{\x}{\y}
    \drawhexi[black!20]{\x}{\y}
  }
}
\foreach \y in {0,...,8}{
  \drawhex[black!20]{4}{\y}
}

\foreach \x/\y in {0/1,2/3,1/-3,1/-2,2/-1,2/0,2/1,2/2,1/4}{
  \draw [thick,fill=black!80] (\poshex{\x}{\y}) circle (0.5cm);
}
\foreach \x/\y in {0/1,2/3,1/-2,2/-1,1/3,0/4}{
  \draw [thick,fill=black!80] (\poshexi{\x}{\y}) circle (0.5cm);
}
\foreach \x/\y in {0/-1,0/2,2/4,2/5,4/2,1/-1,1/5}{
  \draw [thick,fill=black!0] (\poshex{\x}{\y}) circle (0.5cm);
}
\foreach \x/\y in {0/2,2/4,1/0,1/1,1/2,3/2,2/2}{
  \draw [thick,fill=black!0] (\poshexi{\x}{\y}) circle (0.5cm);
}
\end{tikzpicture}
\end{center}
\caption{A winning position for black on a $9 \times 9$ Hex board.}
\label{fig:finite-hex}
\end{figure}

The theory of \emph{infinite games} concerns extensions of board games such as chess to boards of infinite size, possibly with an infinite number of pieces, and possibly played for an infinite or even transfinite number of turns.
For example, recent results in infinite chess include \cite{BrHaSc12,EvHaPe17}.

In the preprint \cite{HaLe23}, Hamkins and Leonessi introduce an infinite variant of Hex, fittingly called \emph{infinite Hex}.
It is played on an infinite board of hexagonal tiles, which we can model as a partition of the plane $\R^2$ into regular hexagons.
The two players still alternate placing a stone of their color in any vacant tile.
Denoting by $T$ the set of all tiles on the board, a position of the game can be modeled by a coloring $x : T \to \{\mathrm{black}, \mathrm{white}, \mathrm{vacant}\}$ that tells whether a given tile contains a black stone, a white stone, or no stone.
We usually say that the tile itself is colored black or white.

The winning condition for the black player is as follows.
A position $x$ is \emph{winning for black} if there exists a two-way infinite sequence $(t_n)_{n \in \Z}$ of tiles with the following properties.
\begin{enumerate}
\item Each tile in the sequence is black: $x(t_i) = \mathrm{black}$ for all $i \in \Z$.
\item Consecutive tiles $t_i$ and $t_{i+1}$ of the sequence are adjacent on the board, for all $i \in \Z$.
\item For each $n \geq 0$ and all large enough $i > 0$, the tile $t_i$ lies within the translated positive quadrant $[n, \infty)^2$ and the tile $t_{-i}$ lies within the translated negative quadrant $(-\infty, -n]^2$.
\end{enumerate}
See Figure~\ref{fig:winning-condition} for a visualization.
The winning condition for white is the same, but with black replaced by white and the quadrants rotated by $90$ degrees, so the path connects $(-\infty, -n] \times [n, \infty)$ and $[n, \infty) \times (-\infty, -n]$.
We let $\mathcal{W}$ be the set of positions that are winning for black.
Since $x \in \mathcal{W}$ only depends on the positions of black tiles in $x$, for simplicity we will assume that all other tiles are vacant.

\begin{figure}[htp]
\begin{center}
\begin{tikzpicture}[scale=0.45]

\draw [very thick, black!80, densely dotted, rounded corners] (-8,-5) -- ++(2,0) -- ++(1,1) -- ++(-2,2) -- ++(1,1) -- ++(2,0) -- ++(-1,-1) -- ++(2,-2) -- ++(3,0) -- ++(2,2) -- ++(-1,1) -- ++(-2,0) -- ++(2,2) -- ++(1,-1) -- ++(1,0) -- ++(1,-1) -- ++(1,0) -- ++(1,-1) -- ++(-2,-2) -- ++(2,0) -- ++(1,1) -- ++(-1,1) -- ++(1,1) -- ++(-1,1) -- ++(-2,0) -- ++(1,1) -- ++(-3,0) -- ++(-3,3) -- ++(1,1) -- ++(3,0) -- ++(2,-2) -- ++(-1,-1) -- ++(3,0) -- ++(-2,2) -- ++(1,0) -- ++(1,1) -- ++(-1,1);

\draw [thick] (3,6) -- (3,3) -- (8,3);
\draw [thick] (-3,-6) -- (-3,-3) -- (-8,-3);
\fill (3,3) circle (0.15cm);
\fill (-3,-3) circle (0.15cm);
\node [above right] at (-3,-3) {$(-n,-n)$};
\node [left] at (3,3) {$(n,n)$};

\end{tikzpicture}
\end{center}
\caption{Black wins if some two-way infinite black path eventually goes from every southwest quadrant to every northeast quadrant.}
\label{fig:winning-condition}
\end{figure}
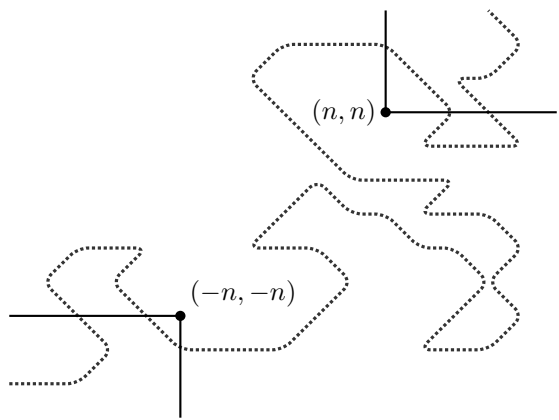

As the set $\mathcal{W}$ is defined by a second-order condition, it is not clear whether it is a Borel subset of all positions.
Hamkins and Leonessi state this as an explicit open problem in \cite{HaLe23}.
A game with a Borel winning condition is well behaved from a set-theoretic standpoint: the Borel determinacy theorem, which is provable in Zermelo-Fraenkel set theory without the axiom of choice \cite{Ma75}, states that starting from any position, either one player has a winning strategy or both players have drawing strategies (if the game admits draws).
Using the axiom of choice one can construct \emph{non-determined} games that admit no winning or drawing strategies for either player.
Hamkins and Leonessi prove in \cite{HaLe23} that starting from the empty board, infinite Hex is a draw with perfect play, but it is \emph{a priori} unknown if other starting positions are even determined.

Our main theorem states that infinite Hex is indeed a Borel game, and thus admits winning or drawing strategies from any starting position within ZF set theory.
In fact, it is arithmetic for the lightface hierarchy.

\begin{theorem}
\label{thm:upper-bound}
  The set $\mathcal{W}$ of colorings of $T$ on which black wins is $\Delta^0_5$ on the lightface hierarchy, and hence $\mathbf{\Delta}^0_5$ on the boldface hierarchy.
\end{theorem}

A weaker version of Theorem~\ref{thm:upper-bound}, giving an upper bound of $\Sigma^0_7$, was published by the author as an answer to a question of Hamkins on MathOverflow \cite{MO}.

\begin{corollary}
  It is provable in ZF set theory that for every initial position of infinite Hex, either one player has a winning strategy or both players have drawing strategies.
\end{corollary}

Conversely, we prove a lower bound for the position of $\mathcal{W}$ on this hierarchy.

\begin{theorem}
\label{thm:lower-bound}
  The set $\mathcal{W}$ of colorings of $T$ on which black wins is $\mathbf{\Sigma}^0_4$-hard for Wadge reductions on the boldface hierarchy, and $\Sigma^0_4$-hard for Turing reductions on the lightface hierarchy.
\end{theorem}

\section{Definitions and notation}

\subsection{Infinite Hex and discrete geometry}

Fix a regular hexagonal tiling of $\R^2$ on which some of the edges are horizontal, and let $T$ be the set of tiles in this tiling.
A coloring over $T$ is a function $x : T \to \{0,1\}$ where $x(t) = 0$ means that tile $t$ is vacant, and $x(t) = 1$ means that it is black.
We fix some natural order on $T$, so that the set of colorings on $T$ is computably isomorphic to $\{0,1\}^\N$.
Especially in Section~\ref{sec:upper}, we have an implicit coloring $x \in \{0,1\}^T$.
In this case we let $B \subset T$ be the set of black tiles, and $N = T \setminus R$ the vacant tiles in $x$.

We say two tiles of $T$ are \emph{adjacent} or \emph{neighbors} if they share an edge.
For $X \subset T$ and $a \in X$, write $C(X,a) \subset X$ for the connected component of $a$ in $X$.
Write $\mathcal{C}(X)$ for the set of all connected components of $X$.

An \emph{edge} of a set $X \subset T$ is a pair $\beta = (a,b)$ of adjacent tiles $a \in X, b \in T \setminus X$.
It represents the actual edge between a tile in $X$ and a tile outside $X$.
The set of edges of $X$ is denoted by $E(X)$.
Each edge $\beta \in E(X)$ defines a two-way infinite sequence $(\beta_i)_{i \in \Z} = (a_i,b_i)_{i \in \Z}$ of edges of $X$ by following them in a consistent direction, say with the $X$-part on the left when going in the positive direction of $\Z$.
Note that this sequence can be periodic.
See Figure~\ref{fig:edge-seq} for an example.

\begin{figure}[htp]
\begin{center}
\begin{tikzpicture}[scale=1.2]

\clip (-2.8,-2.25) rectangle (6.7,3.3);

\foreach \x in {-3,...,2}{
  \foreach \y in {-3,...,2}{
    \drawhex{\x}{\y}
    \drawhexi{\x}{\y}
  }
}

\foreach \x/\y in {0/0,0/1,0/2,1/0}{
  \draw [fill=black!80] (\poshex{\x}{\y}) circle (0.5cm);
}
\foreach \x/\y in {0/0,0/-2}{
  \draw [fill=black!80] (\poshexi{\x}{\y}) circle (0.5cm);
}

\begin{scope}[xshift=0.5 cm+0.5*\coss cm,yshift=-0.5*\sins cm]
\node[draw,fill=white] at (\poshexi{-1}{1}) {$\beta_{-4}$};
\end{scope}
\begin{scope}[xshift=-0.5 cm-0.5*\coss cm,yshift=-0.5*\sins cm]
\node[draw,fill=white] at (\poshex{0}{1}) {$\beta_{-3}$};
\end{scope}
\begin{scope}[xshift=0.5 cm+0.5*\coss cm,yshift=-0.5*\sins cm]
\node[draw,fill=white] at (\poshexi{-1}{0}) {$\beta_{-2}$};
\end{scope}
\begin{scope}[xshift=-0.5 cm-0.5*\coss cm,yshift=-0.5*\sins cm]
\node[draw,fill=white] at (\poshex{0}{0}) {$\beta_{-1}$};
\end{scope}
\begin{scope}[yshift=-\sins cm]
\node[draw,fill=white] at (\poshex{0}{0}) {$\beta_0$};
\end{scope}
\begin{scope}[xshift=0.5 cm+0.5*\coss cm,yshift=-0.5*\sins cm]
\node[draw,fill=white] at (\poshex{0}{0}) {$\beta_1$};
\end{scope}
\begin{scope}[yshift=-\sins cm]
\node[draw,fill=white] at (\poshexi{0}{0}) {$\beta_2$};
\end{scope}
\begin{scope}[xshift=-0.5 cm-0.5*\coss cm,yshift=-0.5*\sins cm]
\node[draw,fill=white] at (\poshex{1}{0}) {$\beta_3$};
\end{scope}
\begin{scope}[yshift=-\sins cm]
\node[draw,fill=white] at (\poshex{1}{0}) {$\beta_4$};
\end{scope}
\begin{scope}[xshift=0.5 cm+0.5*\coss cm,yshift=-0.5*\sins cm]
\node[draw,fill=white] at (\poshex{1}{0}) {$\beta_5$};
\end{scope}
\begin{scope}[xshift=-0.5 cm-0.5*\coss cm,yshift=-0.5*\sins cm]
\node[draw,fill=white] at (\poshexi{1}{0}) {$\beta_6$};
\end{scope}
\begin{scope}[yshift=-\sins cm]
\node[draw,fill=white] at (\poshex{1}{1}) {$\beta_7$};
\end{scope}
\begin{scope}[xshift=-0.5 cm-0.5*\coss cm,yshift=-0.5*\sins cm]
\node[draw,fill=white] at (\poshex{1}{1}) {$\beta_8$};
\end{scope}
\begin{scope}[yshift=-\sins cm]
\node[draw,fill=white] at (\poshexi{0}{1}) {$\beta_9$};
\end{scope}
\begin{scope}[xshift=-0.5 cm-0.5*\coss cm,yshift=-0.5*\sins cm]
\node[draw,fill=white] at (\poshexi{0}{1}) {$\beta_{10}$};
\end{scope}

\end{tikzpicture}
\end{center}
\caption{A sequence of edges in $E(B)$.}
\label{fig:edge-seq}
\end{figure}
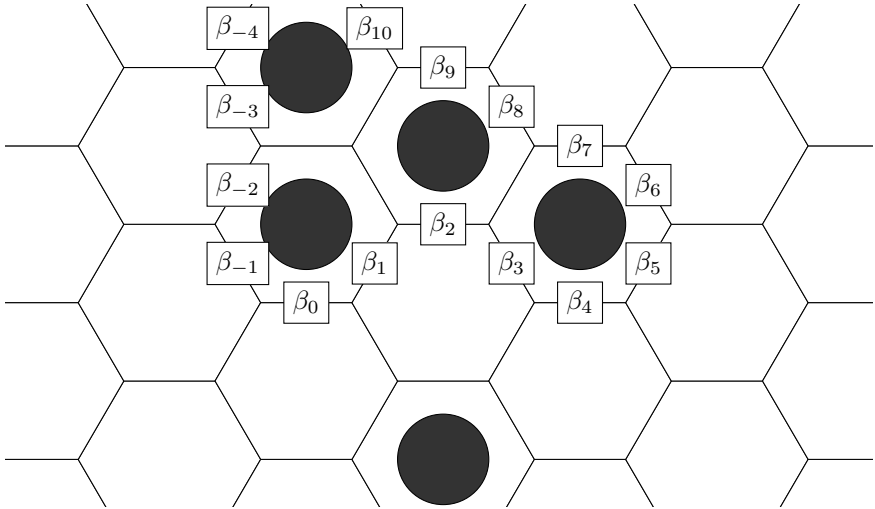

Fix some sequence $(d_n)_{n \in \Z}$ of tiles such that each $d_n$ is the northeast neighbor of $d_{n-1}$.
The \emph{northeast quarter-plane} $Q^+_n \subset T$ contains all tiles of $T$ that are obtained from $d_n$ by translations along vectors in $[0, \infty)^2$.
Symmetrically, the \emph{southwest quarter-plane} $Q^-_n$ contains translates of $d_n$ by vectors in $(-\infty, 0]^2$.
We will use ${*}$ to stand for either ${+}$ or ${-}$.
The \emph{inner border} of $Q^*_n$, denoted $\partial Q^*_n$, consists of those tiles of $Q^*_n$ that have a neighbor in $T \setminus Q^*_n$.
The vertical part of the inner border, denoted $\partial_V Q^*_n$, consists of $d_n$ and all tiles directly to the north of it.
The horizontal part is $\partial_H Q^*_n = (\partial Q^*_n \setminus \partial_V Q^*_n) \cup \{d_n\}$.
See Figure~\ref{fig:quarter} for an illustration of these definitions.
A line in $\R^2$ is called a $V$-line if it is vertical, and a $H$-line if it is horizontal.
We will use $D$ to stand for $V$ or $H$, and denote $H^{-1} = V$ and $V^{-1} = H$.

\begin{figure}[htp]
\begin{center}
\begin{tikzpicture}[scale=0.6]

\clip ($(\poshex{-3}{-3})-(0.5,0.5)$) rectangle ($(\poshexi{2}{2})+(0.4,0.5)$);

\foreach \x in {-3,-2}{
  \foreach \y in {-3,...,2}{
    \drawhex{\x}{\y}
    \drawhexi{\x}{\y}
  }
}
\foreach \x in {-1,...,2}{
  \foreach \y in {-3,-2}{
    \drawhex{\x}{\y}
    \drawhexi{\x}{\y}
  }
}
\foreach \x in {-1,...,2}{
  \foreach \y in {-1,...,3}{
    \drawhex[black!20]{\x}{\y}
    \drawhexi[black!20]{\x}{\y}
  }
}

\foreach \n [count=\i from -3] in {-2,0,...,8}{
  \node at (\poshex{\i}{\i}) {$d_{\n}$};
}
\foreach \n [count=\i from -3] in {-1,1,...,9}{
  \node at (\poshexi{\i}{\i}) {$d_{\n}$};
}

\begin{scope}[yshift=-\sins*0.7 cm]
\foreach \x in {-1,...,2}{
  \node at (\poshex{\x}{-1}) {H};
  \node at (\poshexi{\x}{-1}) {H};
}
\end{scope}
\begin{scope}[yshift=\sins*0.65 cm]
\foreach \y in {-1,...,2}{
  \node at (\poshex{-1}{\y}) {V};
}
\end{scope}

\end{tikzpicture}
\end{center}
\caption{The quarter-plane $Q^+_2$ (in gray) and its borders $\partial_H Q^+_2$ and $\partial_V Q^+_2$ (marked with H and V).}
\label{fig:quarter}
\end{figure}
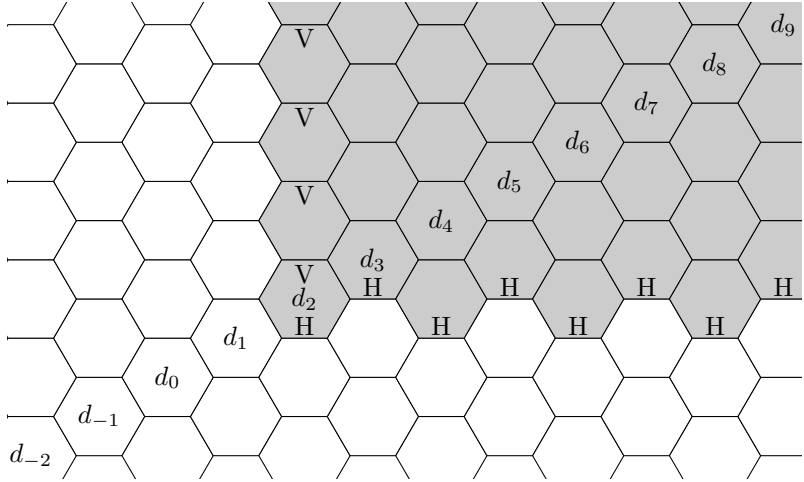

A \emph{path} is a sequence of tiles $(t_i)_{i \in I}$ for a finite or infinite interval $I \subset \Z$ such that whenever $i, i+1 \in I$, the tiles $t_i$ and $t_{i+1}$ are adjacent.
The path is \emph{black} if $t_i \in B$ for all $i \in I$, and \emph{vacant} if $t_i \in N$ for all $i \in I$.
Fix an arbitrary tile $t_0 \in T$ as the ``origin'', and for $r \geq 0$, let $B_r \subset T$ be the set of tiles that are reachable from $t_0$ by a path of length at most $r$.

A one-way infinite path $(t_i)_{i \in \N}$ \emph{reaches to ${*}$}, if for all $n$ there exists $i \in \N$ such that $t_j \in Q^*_n$ holds for all $j \geq i$.
A coloring $x$ is \emph{winning for black} if there exist black paths $(t_i)_{i \in \N}$ and $(s_i)_{i \in \N}$ with $t_0 = s_0$ such that $(t_i)$ reaches to ${+}$ and $(s_i)$ reaches to ${-}$.
This is equivalent to the existence of a single a two-way infinite black path $(t_i)_{i \in \Z}$ such that for all $n$ there exists $i \in \N$ such that $t_j \in Q^+_n$ and $t_{-j} \in Q^-_n$ for all $j \geq i$.
We call such a path a \emph{winning path} for black.
The set of colorings $x \in \{0,1\}^T$ that are winning for black is denoted $\mathcal{W}$.

\subsection{Descriptive set theory}

We recall some notions from descriptive set theory, a good general reference for which is \cite{Ke95}.
Let $X = \{0,1\}^\N$ be the Cantor space.
For each countable ordinal $\alpha \geq 1$, we define the classes $\mathbf{\Sigma}^0_\alpha$, $\mathbf{\Pi}^0_\alpha$ and $\mathbf{\Delta}^0_\alpha$ as follows:
\begin{itemize}
\item $\mathbf{\Sigma}^0_1 = \{ A \subset X : A \text{~is open} \}$.
\item $\mathbf{\Sigma}^0_\alpha = \{ \bigcup_{n \in \N} A_n : \alpha_n < \alpha, A_n \in \mathbf{\Pi}^0_{\alpha_n} \}$ for $\alpha > 1$.
\item $\mathbf{\Pi}^0_\alpha = \{ X \setminus A : A \in \mathbf{\Sigma}^0_\alpha \}$.
\item $\mathbf{\Delta}^0_\alpha = \mathbf{\Sigma}^0_\alpha \cap \mathbf{\Pi}^0_\alpha$.
\end{itemize}
These classes form the \emph{Borel hierarchy} or \emph{boldface hierarchy} of $X$.
The union of $\mathbf{\Sigma}^0_\alpha$ over all countable ordinals $\alpha$ consists of exactly the Borel subsets of $X$.

The \emph{effective Borel hierarchy} or \emph{lightface hierarchy} of the classes $\Sigma^0_\alpha$, $\Pi^0_\alpha$ and $\Delta^0_\alpha$ is defined similarly, but the unions that define the $\Sigma^0_\alpha$ should be computable.
Namely, $\Sigma^0_1$ is the class of \emph{effectively open} sets, i.e.\ sets $A = \bigcup_n A_n$, where $(A_n)_{n \in \N}$ is a computable sequence of clopen sets, and $\Sigma^0_\alpha$ consists of unions $\bigcup_n A_n$, where $(A_n, \alpha_n)_{n \in \N}$ is a computable sequence with $\alpha_n < \alpha$ and $A_n \in \Pi^0_\alpha$.\footnote{There are certain subtleties regarding the way in which the sets $A_n$ and ordinals $\alpha_n$ are encoded, but we skip them since we only handle finite ordinals.}
In particular, if a set $A \subset X$ has the form
\[
A = \{ x \in X : \forall n_1 \exists n_2 \cdots Q n_k \enspace \phi(n_1, n_2, \ldots, n_k, x) \}
\]
with $k$ alternating quantifiers and $\phi$ a computable predicate that, for each choice of the $n_i$, only depends on $x_{[0, n]}$ for some finite $n \in \N$, then $A$ is $\Pi^0_k$.
If the alternation begins with $\exists$ instead of $\forall$, then $A$ is $\Sigma^0_k$.

Let $A, B \subset X$.
A \emph{Wadge reduction} of $A$ to $B$ is a continuous function $f : X \to X$ with $f^{-1}(B) = A$.
If $f$ is computable (in the sense that some oracle Turing machine, given $x \in X$ as an oracle and $i \in \N$ as input, computes $f(x)_i$), it is a \emph{Turing reduction}.
A set $A$ is \emph{hard} for a class $\mathcal{C}$ with respect to Wadge or Turing reductions, if for every $B \in \mathcal{C}$ there exists a reduction of the appropriate type of $B$ to $A$.
If $A$ is hard for $\mathbf{\Sigma}^0_\alpha$ with respect to Wadge reductions, then it is not in $\mathbf{\Pi}^0_\alpha$, and analogously for the lightface hierarchy and Turing reductions.

\section{Upper bound}
\label{sec:upper}

In this section we prove that $\mathcal{W}$ is a $\Delta^0_5$ set by exhibiting a Boolean combination of $\Sigma^0_4$ and $\Pi^0_4$ formulas that defines it.
For this, fix an implicit coloring of $T$, and let $B \subset T$ be the black tiles and $N = T \setminus B$ be the vacant tiles.

We define a few auxiliary formulas, starting with the following.
\begin{align*}
  \psi_1(a,n,{*}) = {} & \forall m : \enspace |C(Q^*_n \cap B, a)| \geq m \\
  \phi_1 = {} & \exists t \in B \enspace \forall n \enspace \exists a \in Q^+_n \cap C(B, t), b \in Q^-_n \cap C(B, t) : \\
  {} & \psi_1(a,n,{+}) \wedge \psi_1(b,n,{-})
\end{align*}
The $\Sigma^0_4$ formula $\phi_1$ states that for some black \emph{anchor tile} $t$, for all $n$ there exists a finite black path connecting an infinite black component of $Q^+_n$ to an infinite black component of $Q^-_n$ via $t$.

\begin{remark}
  \label{rem:not-enough}
  The condition $\phi_1$ in itself is not enough to guarantee a black win.
  In fact, even the following stronger $\Sigma^0_5$ formula does not guarantee a black win:
  \begin{align*}
  \psi'_1(a,n,{*}) = {} & \forall m \enspace \exists c \in Q^*_m : \enspace c \in C(Q^*_n \cap B, a) \\
  \phi'_1 = {} & \exists t \in B \enspace \forall n \enspace \exists a \in Q^+_n \cap C(B, t), b \in Q^-_n \cap C(B, t) : \\
  {} & \psi'_1(a,n,{+}) \wedge \psi'_1(b,n,{-})
  \end{align*}
  This formula states that for some black anchor tile $t$, for all $n$ there exist a black path connecting some $a \in Q^+_n$ and $b \in Q^-_n$ via $t$, such that the connected component $C(Q^+_n \cap B, a)$ intersects every $Q^+_m$, and symmetrically for $b$.
  It implies $\phi_1$, since such a component is necessarily infinite.

  Consider the coloring of $T$ with an infinite black path coming from the southwest followed by an infinite path going directly to the north, with infinitely many paths branching to the east.
  Each branch oscillates to the north and back to the same horizontal line with increasing amplitude.
  See Figure~\ref{fig:comb-example} for an illustration.
  
  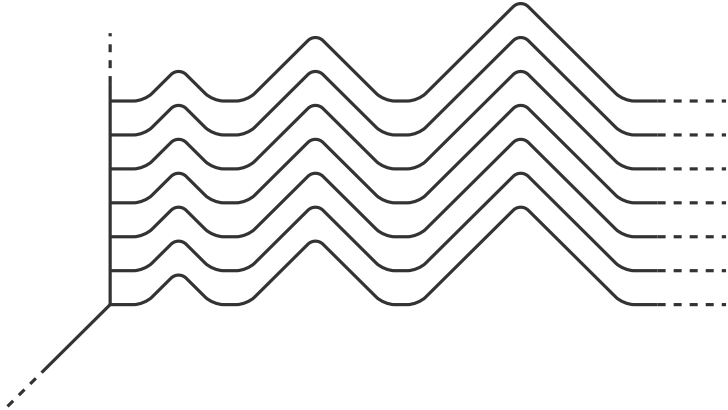
\begin{figure}[htp]
\begin{center}
\begin{tikzpicture}[scale=0.45]

\draw [very thick, black!80, dashed] (-3,-3) -- (-2,-2);
\draw [very thick, black!80] (-2,-2) -- (0,0) -- (0,6.5);
\draw [very thick, black!80, dashed] (0,6.5) -- (0,8);

\foreach \y in {0,...,6}{
  \draw [very thick, black!80, rounded corners] (0,\y) -- ++(1,0) -- ++(1,1) -- ++(1,-1) -- ++(1,0) -- ++(2,2) -- ++(2,-2) -- ++(1,0) -- ++(3,3) -- ++(3,-3) -- ++(1,0);
  \draw [very thick, black!80, dashed] (16,\y) -- ++(2,0);
}


\end{tikzpicture}
\end{center}
\caption{The formula $\phi'_1$ holds, but black does not win.}
\label{fig:comb-example}
\end{figure}
  
  Given $n$, we can choose $a$ from a branch that is wholly contained in $Q^+_n$ and $b$ on the infinite diagonal path.
  Then $\psi'_1(a,n,{+})$ holds, as the branch oscillates to intersect each $Q^+_m$, and $\psi'_1(b,n,{-})$ holds trivially.
  However, no path is winning for black, as every branch returns infinitely often to some horizontal line.
\end{remark}

We present some auxiliary results about the structure of infinite connected components of black tiles.
Together they state that the situation of Remark~\ref{rem:not-enough}, that is, a vertical or horizontal black line with infinitely many branches, is essentially the only obstruction for $\phi_1$ being equivalent to a win for black.

\begin{lemma}
\label{lem:konig}
Suppose that $\phi_1$ holds with anchor tile $t$, and let ${*} \in \{{+},{-}\}$.
Suppose that for all $n$, the set $Q^*_n \cap C(B, t)$ has only finitely many infinite connected components that intersect $\partial Q^*_n$.
Then there exists a black path $(t_i)_{i \in \N}$ with $t_0 = t$ that reaches to ${*}$.
\end{lemma}

\begin{proof}
We may assume ${*} = {+}$ due to symmetry.
Pick an arbitrary $n_1$, and let $a \in Q^+_{n_1} \cap C(B, t)$ be given by $\phi_1$.
Let $n_0 \leq n_1$ be the greatest integer such that $a \in C(B \cap Q^+_{n_0}, t)$.

Denote by $\mathcal{C}_n$ the set of infinite connected components of $Q^+_n \cap C(B \cap Q^+_{n_0}, t)$ that intersect $\partial Q^+_n$.
We define a directed graph $G = (V, E)$ with vertex set $V = \bigcup_{n \geq n_0} \mathcal{C}_n$.
The edge set $E$ is defined by adding an edge from $X \in \mathcal{C}_n$ to $Y \in \mathcal{C}_{n+1}$ if $Y \subset X$.
As $G$ does not contain (directed or undirected) cycles, it is a forest.

Let $X_0 = C(B \cap Q^+_{n_0}, t)$ be the connected component of $t$ in $B \cap Q^+_{n_0}$.
By our choice of $n_0$, we have $b \in X_0$.
Also, $X_0$ is clearly a connected component of $Q^+_{n_0} \cap C(B \cap Q^+_{n_0}, t)$, and because $\psi_1(a, n_1, {+})$ holds, $X_0$ is infinite.
Because of the maximality of $n_0$, it intersects $\partial Q^+_{n_0}$ as well.
Hence $X_0 \in \mathcal{C}_{n_0} \subset V$.
In fact, every $X \in V$ is a subset of $X_0$, hence $G$ is a tree and $X_0$ is its root.

Let $n \geq n_0$ be arbitrary, and choose $a_n \in Q^+_n \cap C(B, t)$ for it by $\phi_1$.
Let $X = C(Q^+_n \cap B, a_n)$ be its connected component.
Since $\psi_1(a_n, n, {+})$ holds, $X$ is infinite, and since $a_n$ is connected to $t$ via black tiles, $X$ also intersects $\partial Q^+_n$.
Hence $X \in \mathcal{C}_n \subset V$.
In particular, $G$ is an infinite tree.

Because of the assumption of the lemma, each $\mathcal{C}_n$ is finite.
Hence every vertex of $G$ has finite degree.
By König's lemma, $G$ has an infinite branch $X_0, X_1, X_2, \ldots$.
We choose a tile $t_n \in X_n$ from each of these sets with $t_0 = t$.
By construction, each $t_n$ can be connected to $t_{n+1}$ by a black path within $Q^+_n \cap B$.
The concatenation of these paths produces a one-way infinite black path that starts from $t$ and lies eventually in every $Q^+_n$, i.e.\ reaches to ${+}$.
\end{proof}

\begin{lemma}
\label{lem:comb-exists}
Suppose that $\phi_1$ holds with anchor $t$, but there does not exist a path of black tiles that starts from $t$ and reaches to ${*}$.
Then for some $n$ and $D \in \{H,V\}$, the set $\partial_D Q^*_n$ contains an infinite number of edges $\beta \in E(B)$ that are connected to $t$ by black tiles, such that the one-directional edge sequence that begins at $\beta$ and goes in the interior of $Q^*_n$ never leaves $Q^*_n$. 
\end{lemma}

\begin{proof}
Because of Lemma~\ref{lem:konig}, there exists $n$ such that $Q^*_n \cap C(B, t)$ has infinitely many infinite connected components that intersect $\partial Q^*_n$, each of which is also a component of $Q^*_n \cap B$.
For some $D \in \{H,V\}$, infinitely many of these components intersect $\partial_D Q^*_n$.
Each such component $X$ (except possibly one that contains the corner of $Q^*_n$) contains a tile $a \in \partial_D Q^*_n$ that is closest to the corner of $Q^*_n$ and part of an edge $\beta = (a,b) \in E(B)$ with $b \in \partial_D Q^*_n$.
Since $X$ is infinite, the edge sequence $(\beta_i)$ that begins at $\beta$ and continues to the interior of $Q^*_n$ must be infinite.
\end{proof}

Next, we define the formula
\begin{align*}
  \phi_2({*}) = {} & \forall n \enspace \exists r \enspace \forall (a,b) \in E(B) : \\
  {} & a, b \in \partial Q^*_n \setminus B_r \implies \\
  {} & \exists i < 0, j > 0 : \enspace a_i, b_i, a_j, b_j \notin Q^*_n
\end{align*}
This $\Pi^0_4$ formula states that for all $n$, all but finitely many edge sequences that cross the border of $Q^*_n$ will leave $Q^*_n$ after finitely many steps, in either direction.
Recall that $B_r$ is the ball of radius $r$ around the ``origin tile'' in the path metric, and here it just allows us to ignore a finite subset of the edge sequences.
The formula $\phi_2({*})$ is equivalent to the condition that for all $n$, there are only finitely many infinite connected components of $Q^*_n \cap B$ that intersect $\partial Q^*_n$, but we do not actually need this fact.

The following lemma is quite technical, but we use it several times.

\begin{lemma}
\label{lem:no-room}
Suppose that $\phi_2({*})$ does not hold: for some $n$ and $D \in \{H,V\}$, there are infinitely many edges $\beta \in E(B)$ contained in $\partial_D Q^*_n$ such that the one-directional edge sequence that begins at $\beta$ and and goes inside $Q^*_n$ never leaves $Q^*_n$.
Then each of these edge sequences intersects any given $D$-line only finitely many times.
\end{lemma}

\begin{proof}
By symmetry, we may assume ${*} = {+}$ and $D = V$.
Let $\beta^k$ for $k \in \N$ be the edges given by the assumption, ordered by distance to the corner of $Q^+_n$, and let $(\beta^k_i)_i$ be the associated edge sequences.

Suppose for a contradiction that one of the edge sequences $(\beta^k_i)_i$ visits some vertical line $\ell$ infinitely often.
Let $d$ be the distance between $\partial_V Q^+_n$ and $\ell$, measured in columns of tiles.
Consider the next $d$ edges $\beta^{k+1}, \ldots, \beta^{k+d}$ to the north of $\beta^k$, and let $\ell'$ be a horizontal line to the north of $\beta^{k+d}$.

There exists an $i$ such that tile $\beta^k_i$ of the edge sequence of $\beta^k$ intersects $\ell$ somewhere to the north of $\ell'$.
Now the $d$ infinite edge sequences $(\beta^{k+p}_j)_j$ for $p = 1, \ldots, d$ have to fit through the width-$d$ gap between $\partial Q^+_n$ and $\ell$, which is impossible.
See Figure~\ref{fig:no-room} for an illustration.
Hence $(\beta^k_i)_i$ visits each vertical line only finitely many times.
\end{proof}

\begin{figure}[htp]
\begin{center}
\begin{tikzpicture}[xscale=1.5]

\clip (-1,-2) rectangle (4,5);

\draw [thick, black!80, rounded corners] (0.5,-0.7) -- (3.1,-0.7) -- (3.1,3.5) -- (1.5,3.5) -- (1,3.7) -- (1,4.3) -- (1.5,4.5) -- (6,4.5) -- (6,5);

\draw [thick, black!80, rounded corners] (0.5,-0.2) -- (2.7,-0.2) -- (2.7,3) -- (0.9,3.5) -- (0.9,6);
\draw [thick, black!80, rounded corners] (0.8,6) -- (0.8,3.3) -- (2.5,2.9) -- (2.5,0.3) -- (0.5,0.3);
\draw [thick, black!80, rounded corners] (0.5,0.8) -- (2,0.8) -- (2,2.9) -- (0.7,3.2) -- (0.7,6);
\draw [thick, black!80, rounded corners] (0.6,6) -- (0.6,3.1) -- (1.8,2.8) -- (1.8,1.3) -- (0.5,1.3);
\draw [thick, black!80, rounded corners] (0.5,1.8) -- (1.5,1.8) -- (1.5,2.5) -- (0.7,2.8);

\draw (0.5,5) -- (0.5,-1) -- (7,-1);
\draw [densely dotted] (1,5) -- (1,-1.5);
\draw [densely dotted] (0,2.5) -- (7,2.5);

\node [left] at (0.5,4.2) {$\partial_V Q^+_n$};
\node [right] at (1,4) {$\beta^k_i$};
\node [left] at (0.5,-0.7) {$\beta^k$};
\node [left] at (0.5,-0.2) {$\beta^{k+1}$};
\node [left] at (0.5,0.3) {$\beta^{k+2}$};
\node [left] at (0.5,0.8) {$\beta^{k+3}$};
\node [left] at (0.5,1.3) {$\beta^{k+4}$};
\node [left] at (0.5,1.8) {$\beta^{k+d}$};
\node [below] at (1,-1.5) {$\ell$};
\node [left] at (0,2.5) {$\ell'$};

\end{tikzpicture}
\end{center}
\caption{The edge sequences $(\beta^{k+p}_j)_j$ are not able to squeeze between $\partial_V Q^+_n$ and $\ell$.}
\label{fig:no-room}
\end{figure}
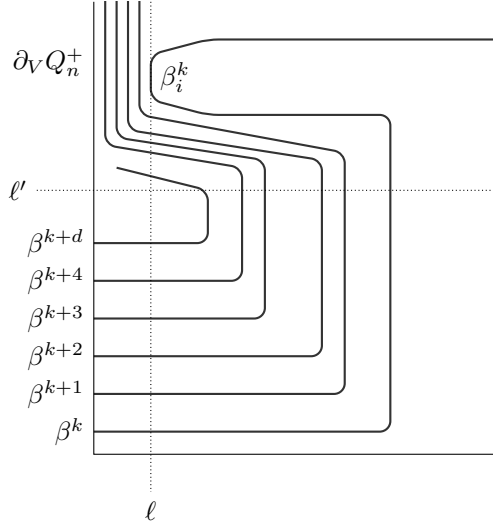

Finally, we define
\begin{align*}
  \phi_3({*}) = {} & \exists (a,b) \in E(B) \enspace \forall n \enspace \exists k : \\
  {} & (\forall i \geq 0 : a_{k+i} \in Q^*_n) \vee (\forall i \geq 0 : a_{k-i} \in Q^*_n)
\end{align*}
This $\Sigma^0_4$ formula states that there exists a black edge such that the path it defines eventually stays in each quarter-plane $Q^*_n$ (and in particular is infinite).

\begin{lemma}
\label{lem:no-border}
Suppose that black wins.
Then $\phi_2({*}) \vee \phi_3({*})$ holds for each ${*} \in \{{+},{-}\}$.
\end{lemma}

\begin{proof}
Let $(t_i)_{i \in \Z}$ be the winning path for black.
We assume that $\phi_2({+})$ does not hold, and prove $\phi_3({+})$; the case for ${*} = {-}$ is symmetric.
Hence, for some $n$ and $D \in \{H,V\}$, the border $\partial_D Q^+_n$ contains infinitely many edges $\beta^k \in E(B)$ such that the one-way infinite edge sequence $(\beta^k_i)_i$ that goes inside $Q^+_n$ never leaves $Q^+_n$.
Again by symmetry, we may assume $D = V$.

Let $i_0 \in \Z$ be the lowest index such that $t_j \in Q^+_n$ for all $j \geq i_0$.
Choose a $\beta^k$ that intersects $\partial Q^+_n$ to the north of $t_{i_0}$.
By Lemma~\ref{lem:no-room}, the sequence $(\beta^k_i)_i$ visits any given vertical line only finitely many times.
It cannot visit any horizontal line infinitely many times, as this would prevent $(t_j)_{j \geq i_0}$ from reaching to ${+}$.
Hence it is a witness for $\phi_3({+})$.
\end{proof}

We claim that the following first order formula characterizes the winning condition for black:
\[
  \phi_4 = \phi_1 \wedge (\phi_2({+}) \vee \phi_3({+})) \wedge (\phi_2({-}) \vee \phi_3({-}))
\]
As a Boolean combination of $\Sigma^0_4$ and $\Pi^0_4$ formulas, $\phi_4$ defines a $\Delta^0_5$ set of colorings of $T$.

\begin{lemma}
The formula $\phi_4$ is equivalent to a black win.
\end{lemma}

\begin{proof}
Clearly a black win implies $\phi_1$, as we can choose $a$, $b$ and $c$ from the winning path.
By Lemma~\ref{lem:no-border}, a black win implies $\phi_2({+}) \vee \phi_3({+})$ and $\phi_2({-}) \vee \phi_3({-})$ as well.
Thus a black win implies $\phi_4$.

Conversely, suppose for a contradiction that $\phi_4$ holds with anchor tile $t \in B$ but black does not win.
By symmetry, we may assume there does not exist a path of black cells that starts from $t$ and reaches to ${+}$.

By Lemma~\ref{lem:comb-exists}, for some $n_0$ and $D \in \{H,V\}$ there exist infinitely many one-way infinite edge sequences $(\beta^k_i)_i$, each of which originates from $\partial_D Q^+_{n_0}$, never leaves $Q^+_{n_0}$ and is connected to $t$ by black tiles.
By symmetry, we may assume $D = V$.
In particular, $\phi_2({+})$ does not hold, so $\phi_3({+})$ must hold: there exists an edge sequence $\beta = (\beta_i)_{i \in \Z}$ and $s = \pm 1$ such that for all $n$ we have $\beta_{s i} \in Q^+_n$ for all large enough $i$.
Note that this does not immediately give a contradiction, since the anchor $t$ might not be connected to $\beta$ by black tiles.

Eventually $\beta_{s i}$ lies in $Q^+_{n_0}$.
Thus, it is below one of the edge sequences given by Lemma~\ref{lem:comb-exists}, say $(\beta^k_i)_i$, in the sense that $(\beta^k_i)_i$ separates $Q^+_{n_0}$ into two components and each edge $\beta_{s i}$ is in the component that contains the corner of $Q^+_{n_0}$.
By Lemma~\ref{lem:no-room}, the sequence $(\beta^k_i)_i$ visits any given vertical line only finitely many times, and since $(\beta_{s i})_{i \in \N}$ reaches to ${+}$, $(\beta^k_i)_i$ visits any given horizontal line only finitely often too.
As $\beta^k$ is connected to $t$ by black tiles, we have constructed a black path that starts at $t$ and reaches to ${+}$, which is a contradiction.
Thus $\phi_4$ implies a black win.
\end{proof}

Theorem~\ref{thm:upper-bound} follows immediately.

\section{Lower bound}
\label{sec:lower-bound}

In this section we prove Theorem~\ref{thm:lower-bound}.
More explicitly, we show that $\mathcal{W}$ is $\Sigma^0_4$-hard in the lightface hierarchy in terms of Turing reductions.
By relativizing the reduction, $\mathbf{\Sigma}^0_4$-hardness in the boldface hierarchy in terms of Wadge reductions follows.

Consider an arbitrary set
\[
X = \bigcup_{a \in \N} \bigcap_{b \in \N} \bigcup_{c \in \N} \bigcap_{d \in \N} X(a,b,c,d)
\]
in the lightface $\Sigma^0_4$ class, so that each $X(a,b,c,d) \subset \{0,1\}^\N$ is a clopen set computable from the numbers $a$, $b$, $c$ and $d$.
Given a sequence $x \in \{0,1\}^\N$, we will effectively construct a coloring $y = f(x) \in \{0,1\}^T$ such that $y \in \mathcal{W}$ if and only if $x \in X$.
As a first step, we note that one can produce clopen sets $X(a,b,c)$ computable from the parameters such that $\bigcup_{c \in \N} \bigcap_{d \in \N} X(a,b,c,d) = \bigcup_{k \in \N} \bigcap_{n \geq N} X(a,b,n)$ for all $a, b \in \N$.
Namely, we choose $X(a,b,n)$ as the set of those $x \in X$ for which there exists a partition $n = i_1 + \cdots + i_k$ such that
\[
x \in \bigcap_{c=1}^k \left( \bigcap_{d = 0}^{i_c-1} X(a,b,c,d) \right) \setminus X(a,b,c,i_c).
\]

As in Remark~\ref{rem:not-enough}, the coloring $y$ contains an infinite black path extending from the origin to the southwest, and another black path $L = (t_i)_{i \geq 0}$ extending directly to the north, where the center of $t_0$ is the origin.
From the vertical path $L$, an infinite number of black paths $(P_i)_{i \in \N}$ branch to the east, so that the first tile of path $P_i$ is tile $t_{2i}$ of $L$.
Denote by $k_i$ the horizontal line that crosses the center of the tile $t_{2i}$.
Define also vertical lines $\ell_j$ for $j \in \N$ such that $\ell_0$ passes through the origin and $\ell_{j+1}$ lies $2j+4$ columns to the right of $\ell_j$.
For each $j \in \N$, the path $P_i$ will pass through the tile $t_{(i,j)}$ containing the intersection point of $\ell_j$ and $k_{i + j}$.
Between $t_{(i,j)}$ and $t_{(i,j+1)}$ it will descend to intersect some line $k_{i+j'}$ with $0 \leq j' \leq j$.
This segment of $P_i$ (or the collection of these segments for all the paths) is called \emph{step $j$}.
Note that $P_i$ can descend to $k_{i+j'}$ on step $j$ only if for all $i' < i$, the path $P_{i'}$ also descends to $k_{i'+j'}$ or below on step $j$.
Then black wins if and only if some path $P_i$ descends to any given line only on a finite number of steps.
See Figure~\ref{fig:lower-bound-structure} for an illustration.

\begin{figure}[htp]
  \begin{center}
    \begin{tikzpicture}[scale=0.2]

      \node [right] at (\poshexi{14}{2}) {$P_0$};
      \node [right] at (\poshexi{14}{4}) {$P_1$};
      \node [right] at (\poshexi{14}{6}) {$P_2$};
      \node [right] at (\poshexi{14}{8}) {$P_3$};
      \node [right] at (\poshexi{14}{10}) {$P_4$};
      \node [above] at (\poshex{0}{11}) {$L$};

      \node [below] at (\poshex{0}{-3}) {$\ell_0$};
      \node [below] at (\poshex{2}{-3}) {$\ell_1$};
      \node [below] at (\poshex{8}{-3}) {$\ell_2$};
      \node [left] at (\poshex{-3}{0}) {$k_0$};
      \node [left] at (\poshex{-3}{2}) {$k_1$};
      \node [left] at (\poshex{-3}{4}) {$k_2$};
      \node [left] at (\poshex{-3}{6}) {$k_3$};
      \node [left] at (\poshex{-3}{8}) {$k_4$};
      \node [left] at (\poshex{-3}{10}) {$k_5$};
      
      \clip ($(\poshex{-3}{-3})-(0.5,0.5)$) rectangle ($(\poshexi{14}{10})+(0.4,0.5)$);

      \foreach \x in {-3,...,14}{
        \foreach \y in {-3,...,10}{
          \drawhex{\x}{\y}
          \drawhexi{\x}{\y}
        }
      }

      \foreach \x in {0,2,8}{
        \draw [densely dotted] (\poshex{\x}{-4}) -- (\poshex{\x}{11});
      }
      \foreach \y in {0,2,...,10}{
        \draw [densely dotted] (\poshex{-4}{\y}) -- (\poshex{15}{\y});
      }

      \foreach \x in {-3,-2,-1}{
        \draw [fill=black!80] (\poshex{\x}{\x}) circle (0.5cm);
        \draw [fill=black!80] (\poshexi{\x}{\x}) circle (0.5cm);
      }
      \foreach \y in {0,...,11}{
        \draw [fill=black!80] (\poshex{0}{\y}) circle (0.5cm);
      }
      \foreach \y in {0,2,...,10}{
        \draw [fill=black!80] (\poshexi{0}{\y}) circle (0.5cm);
      }

      \foreach \y/\yi [count=\x] in {1/1,2/1,1/0,0/0,1/1,2/2,3/3,4/3,3/2,2/1,1/0,0/0,1/1,2/2}{
        \pgfmathsetmacro{\yy}{\y+0}
        \pgfmathsetmacro{\yyi}{\yi+0}
        \draw [fill=black!80] (\poshex{\x}{\yy}) circle (0.5cm);
        \draw [fill=black!80] (\poshexi{\x}{\yyi}) circle (0.5cm);
      }
      \foreach \y/\yi [count=\x] in {1/1,2/1,1/0,0/0,1/1,2/2,3/3,4/3,3/2,2/1,1/0,0/0,1/1,2/2}{
        \pgfmathsetmacro{\yy}{\y+2}
        \pgfmathsetmacro{\yyi}{\yi+2}
        \draw [fill=black!80] (\poshex{\x}{\yy}) circle (0.5cm);
        \draw [fill=black!80] (\poshexi{\x}{\yyi}) circle (0.5cm);
      }
      \foreach \y/\yi [count=\x] in {1/1,2/1,2/1,2/1,2/1,2/2,3/3,4/3,3/2,2/1,1/0,0/0,1/1,2/2}{
        \pgfmathsetmacro{\yy}{\y+4}
        \pgfmathsetmacro{\yyi}{\yi+4}
        \draw [fill=black!80] (\poshex{\x}{\yy}) circle (0.5cm);
        \draw [fill=black!80] (\poshexi{\x}{\yyi}) circle (0.5cm);
      }
      \foreach \y/\yi [count=\x] in {1/1,2/1,2/1,2/1,2/1,2/2,3/3,4/3,3/2,2/1,2/1,2/1,2/1,2/2}{
        \pgfmathsetmacro{\yy}{\y+6}
        \pgfmathsetmacro{\yyi}{\yi+6}
        \draw [fill=black!80] (\poshex{\x}{\yy}) circle (0.5cm);
        \draw [fill=black!80] (\poshexi{\x}{\yyi}) circle (0.5cm);
      }
      \foreach \y/\yi [count=\x] in {1/1,2/1,2/1,2/1,2/1,2/2,3/3,4/3,3/2,2/1,2/1,2/1,2/1,2/2}{
        \pgfmathsetmacro{\yy}{\y+8}
        \pgfmathsetmacro{\yyi}{\yi+8}
        \draw [fill=black!80] (\poshex{\x}{\yy}) circle (0.5cm);
        \draw [fill=black!80] (\poshexi{\x}{\yyi}) circle (0.5cm);
      }
      \foreach \y/\yi [count=\x] in {1/1,2/1,2/1,2/1,2/1,2/2,3/3,4/3,3/2,2/1,2/1,2/1,2/1,2/2}{
        \pgfmathsetmacro{\yy}{\y+10}
        \pgfmathsetmacro{\yyi}{\yi+10}
        \draw [fill=black!80] (\poshex{\x}{\yy}) circle (0.5cm);
        \draw [fill=black!80] (\poshexi{\x}{\yyi}) circle (0.5cm);
      }

      \foreach \x [count=\i from 0] in {0,2,8}{
        \foreach \y in {0,2,...,8,10}{
          \pgfmathsetmacro{\yi}{\y+2*\i}
          \draw[fill=white] (\poshex{\x}{\yi}) circle (0.2cm);
        }
      }
      
    \end{tikzpicture}
    \caption{The structure of the coloring $y$. The tiles $t_{(i,j)}$ are marked with white dots.}
    \label{fig:lower-bound-structure}
  \end{center}
\end{figure}
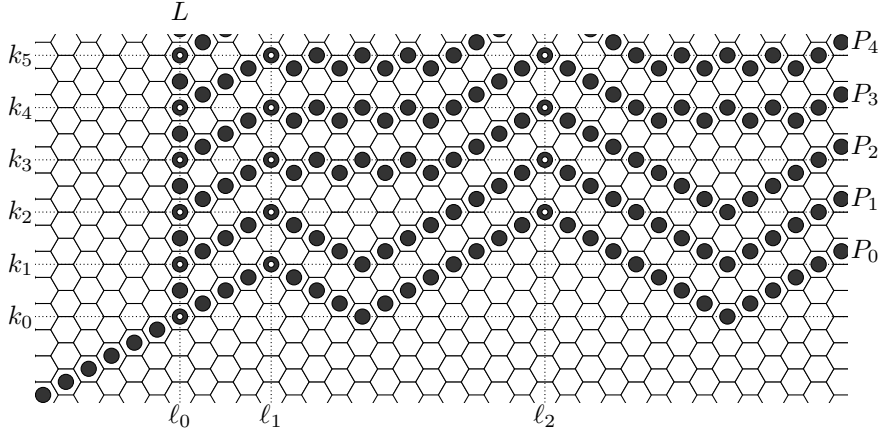

For $i, j, b \in \N$, define $K(i,j) = \{b \leq j : x \notin X(i,b,j)\}$.
The condition $b \in K(i,j)$ should be thought of as ``evidence'' for $x \notin \bigcup_k \bigcap_{n \geq k} X(i,b,n)$, as the latter is equivalent to the condition that $b \in K(i,j)$ for infinitely many $j$.
If $b \in K(i,j)$, we say that path $P_i$ is \emph{commanded to line $k_{i+b}$ on step $j$}.

For each $(i,j) \in \N^2$ we define a finite set $A(i,j) \subset \{0, \ldots, j\}^2$ of integer pairs.
The intuition for $(b,b') \in A(i,j)$ is that path $P_i$ has been commanded to descend to $k_{i+b}$, but up until step $j$ has only managed to descend to $k_{i+b'}$.
First, $A(i,0) = \emptyset$ for all $i \in \N$, since there are no commands before step $0$.
For $j > 0$, define the set $A'(i,j) = A(i,j-1) \cap \{ (b,j) : b \in K(i,j) \}$, which corresponds to the old and new commands of $P_i$.
Let $b(i,j) = \min \{ b : (b,b') \in A'(i,j) \}$ be the relative index of the lowest line that path $P_i$ has been commanded to, and $r(i,j) = \min \{ j - b(i',j) : i' \leq i \}$ the minimum distance that any path in $P_0, \ldots, P_i$ has been commanded to descend.

For each $(b,b') \in A'(i,j)$ we do the following.
\begin{enumerate}
\item If $b' \leq j - r(i,j)$, then we put $(b,b')$ into $A(i, j+1)$.
\item If $b < j - r(i,j) < b'$, then we put $(b, j - r(i,j))$ into $A(i, j+1)$.
\item If $j - r(i,j) \leq b$, then we do not modify $A(i, j+1)$.
\end{enumerate}
This concludes the definition of $A(i, j+1)$.
We define that in the coloring $y$, the path $P_i$ descents to line $k_{j-r(i,j)}$ on step $j$.
This is possible, since for each $j$, the number $j-r(i,j)$ is strictly increasing in $i$.
This concludes the construction of $y = f(x)$.
The function $f$ is computable, since the sets $K(i,j)$ and $A(i,j)$ are computable from $x$.

Suppose that $x \in X$.
Then for some $i \in \N$, every $b \in \N$ lies in only finitely many $K(i,j)$.
Thus, the path $P_i$ is commanded to each line $k_{i+b}$ only finitely many times.
Let $b \in \N$, and suppose $j$ is the last step on which $P_i$ is commanded to any of the lines in $k_i, \ldots, k_{i+b}$.
Then $m_b(j) = \max \{ b'' : b' \leq b, (b', b'') \in A(i,j) \}$ is decreasing in $j \geq j_b$ (here $\max \emptyset = -\infty)$, so on some step $j_b'$ it reaches its minimum value.
After step $j_b'$, the line $P_i$ will not descend to or below line $k_{i+b}$, as doing so would require either being commanded to do so, or for $m_b(j)$ to decrease.
Hence the path $P_i$ reaches to ${+}$, and $y \in \mathcal{W}$.

Suppose then that $x \notin X$: for each $i \in \N$, there exists $b_i \in \N$ that lies in $K(i,j)$ for infinitely many $j$.
Define $b'_0 = b_0$, and for $i > 0$ let $b'_i = \min \{ b_i, b'_{i-1}+1\}$.
We prove by induction on $i$ that each path $P_i$ descends infinitely often to line $k_{i+b'_i}$.
This is true for $i = 0$: every time $P_0$ is commanded to descend to $k_{b'_0}$, it will do so immediately.
For $i > 0$, suppose path $P_i$ is commanded to descend to $k_{i+b_i}$ on step $j$.
Let $j' \geq j$ be the earliest step on which path $P_{i-1}$ descends to $k_{i-1 + b'_{i-1}}$, which exists by the induction hypothesis.
On any step $j \leq j'' < j'$, path $P_i$ is not able to descend to $k_{i + b'_i}$ or below.
Thus the set $A'(i,j')$ contains a pair $(b_i, b')$ with $b' > b'_i$, so $P_i$ descends to $k_{i+b'_i}$ on step $j'$.
This means $y \notin \mathcal{W}$, as none of the paths $P_i$ reaches to ${+}$.

We have proved that $\mathcal{W}$ is $\Sigma^0_4$-hard in terms of Turing reductions.
For Wadge reductions, we simply drop the assumption that the sets $X(a,b,c,d)$ be computable from $a, b, c, d \in \N$.
Each set $K(i,j)$ depends on finitely many indices of $x$, so the function $f$ remains continuous.
The rest of the proof is unaffected, so $f$ is a Wadge reduction.

\section{Conclusion}

We have placed $\mathcal{W}$, the winning condition of infinite Hex, inside the Borel hierarchy, more specifically between $\mathbf{\Sigma}^0_4$ and $\mathbf{\Delta}^0_5$.
However, its exact position remains open.
The fact that the subformulas $\phi_2({*})$ are $\Pi^0_4$ prevents us from placing $\mathcal{W}$ in $\mathbf{\Sigma}^0_4$, and on the other hand, we have not found a reduction for $\mathbf{\Pi}^0_4$-completeness.

Recall that for $n \geq 1$, the class $\mathbf{\Delta}^0_{n+1}$ has no complete sets in terms of Wadge reductions, instead splitting into the \emph{difference hierarchy} over $\mathbf{\Sigma}^0_n$ indexed by the countable ordinals \cite[Section~22.E]{Ke95}.
The finite levels of this sub-hierarchy consist of finite Boolean combinations of $\mathbf{\Sigma}^0_n$ sets.
Thus, if $\mathcal{W}$ is not in $\mathbf{\Sigma}^0_4$, at least it lies relatively low on the difference hierarchy.

\begin{problem}
  What is the exact position of $\mathcal{W}$ in the Borel hierarchy?
\end{problem}

Any tiling $\mathcal{T}$ of $\R^2$ by finite polygonal tiles gives a variant of infinite Hex that is played on $\mathcal{T}$ instead of the regular hexagonal tiling.
The goal of black is to construct a bi-infinite black path $(t_i)_{i \in \Z}$ such that for each $n$, the tile $t_i$ is contained in $[n, \infty)^2$ and $t_{-i}$ is contained in $(-\infty, -n]^2$ for all large enough $i > 0$, and white has the analogous winning condition.
Adjacency of tiles is defined by having a shared edge of positive length, so tiles that overlap only at their corners are not adjacent.
This preserves Theorem~2 of \cite{HaLe23}: at most one player can win.
Let $\mathcal{W}_\mathcal{T}$ be the set of winning positions for black in this variant.

It seems that if $\mathcal{T}$ contains only finitely many congruence classes of tiles, then Theorems~\ref{thm:upper-bound} and~\ref{thm:lower-bound} hold for $\mathcal{W}_\mathcal{T}$ with essentially the same proofs.
However, if $\mathcal{T}$ contains tiles of arbitrarily small diamater, Lemma~\ref{lem:no-room} may be false, as the number of tiles on a shortest path between two lines may be unrelated to their Euclidean distance.
On the other hand, if $\mathcal{T}$ contains arbitrarily large tiles, the construction of section~\ref{sec:lower-bound} may not be possible, as the thickness of each path $P_i$ may increase to the right.

\begin{problem}
  How does the position of $\mathcal{W}_\mathcal{T}$ in the Borel hierarchy depend on the tiling $\mathcal{T}$?
\end{problem}

\section*{Acknowledgements}

The author was supported by the Academy of Finland under grant 346566.

\bibliographystyle{plain}
\bibliography{hexbib}

\end{document}